\documentclass[11pt]{amsart}

\usepackage{amsmath}
\usepackage{fullpage}
\usepackage{xspace}
\usepackage[psamsfonts]{amssymb}
\usepackage[latin1]{inputenc}
\usepackage{graphicx,color}
\usepackage[curve]{xypic} 
\usepackage{hyperref}
\usepackage{graphicx}
\usepackage{amsmath}%
\usepackage{amsthm}%
\usepackage{amscd}
\usepackage{amsfonts}%
\usepackage{amssymb}%
\usepackage{graphicx}

\usepackage{mathrsfs}

\usepackage{tikz}
\usetikzlibrary{matrix,arrows}

\newtheorem{theorem}{Theorem}[section]
\theoremstyle{plain}

\newtheorem{conjecture}[theorem]{Conjecture}
\newtheorem{corollary}[theorem]{Corollary}

\newtheorem{lemma}[theorem]{Lemma}

\numberwithin{equation}{section}

\newcommand{\Acal}{\mathscr{A}}

\newcommand{\Kcal}{\mathscr{K}}
\newcommand{\Lcal}{\mathscr{L}}

\newcommand{\Ocal}{\mathscr{O}}

\newcommand{\Pro}{\mathbb{P}}

\newcommand{\Z}{\mathbb{Z}}
\newcommand{\C}{\mathbb{C}}

\newcommand{\F}{\mathbb{F}}
\newcommand{\Q}{\mathbb{Q}}
\newcommand{\R}{\mathbb{R}}

\newcommand{\Disc}{\mathrm{Disc}}
\newcommand{\rk}{\mathrm{rank}\,}
\newcommand{\supp}{\mathrm{supp}\,}
\newcommand{\Vol}{\mathrm{Vol}\,}

\begin{document}
\title[Bounded ranks]{Bounded ranks and Diophantine error terms}
\author{Hector Pasten}
\address{ Department of Mathematics\newline
\indent Harvard University\newline
\indent Science Center\newline
\indent 1 Oxford Street\newline
\indent Cambridge, MA 02138, USA}
\email[H. Pasten]{hpasten@gmail.com}%
%\thanks{}
%\thanks{This paper is in final form and no version of it will be submitted for publication elsewhere.}
\date{\today}
%\dedicatory{}

\begin{abstract} 

We show that Lang's conjecture on error terms in Diophantine approximation implies Honda's conjecture on ranks of elliptic curves over number fields. We also show that even a very weak version of Lang's error term conjecture would be enough to deduce boundedness of ranks for quadratic twists of elliptic curves over number fields. This can be seen as evidence for boundedness of ranks not relying on probabilistic heuristics on elliptic curves.

\end{abstract}

\maketitle

%\tableofcontents

%%%%%%%%%%%%%%%%%%%%%%%%%%%%%%%%%%%%%%
%%%%%%%%%%%%%%%%%%%%%%%%%%%%%%%%%%%%%%
%%%%%%%%%%%%%%%%%%%%%%%%%%%%%%%%%%%%%%

\section{Introduction}

 In 1960, Honda proposed the following conjecture (cf. p. 98 in \cite{Honda}):
\begin{conjecture}[Honda] \label{ConjHonda} Let  $A$ be an abelian variety over a number field $k$. There is  $c=c(A,k)>0$  such that for every finite extension $L/k$ we have $\rk A(L) \le c\cdot [L:k]$.
\end{conjecture}

 A direct consequence of Honda's conjecture is:
\begin{conjecture}[Uniform boundedness for ranks of quadratic twists] \label{ConjTwist} Let $E$ be an elliptic curve over a number field $k$. There is a positive integer $R=R(E,k)$ satisfying that for every elliptic curve $E'$ over $k$ which is a quadratic twist of $E$, we have $\rk E'(k)\le R$.
\end{conjecture}

The fact that Conjecture \ref{ConjHonda} implies Conjecture \ref{ConjTwist}  has been considered in the past as evidence \emph{against}  Honda's conjecture, see for instance the discussion in p.162 of \cite{Serre}. 

 Uniform boundedness of ranks of elliptic curves over number fields is a somewhat controversial topic. Elliptic curves of fairly large rank have been constructed over $\Q$ \cite{Elkies} and it is known that the ranks of elliptic curves over $\F_p(t)$ are unbounded  \cite{ShTa, Ulmer}.  On the other hand, different probabilistic heuristics have been put forward in the direction of uniform boundedness \cite{ BKLPR, PPVW, PoonenICM, PR, RuSi, WDEFGR}. See Section 3 of \cite{PPVW} for a historical account.  Nowadays, it seems unclear whether Conjecture \ref{ConjTwist} should be expected to be false or not. 

Lang has conjectured improved error terms for Roth's theorem in Diophantine approximation and more generally, in the context of Vojta's conjectures. We will show that Lang's error term conjecture in a very particular setting implies Honda's conjecture for elliptic curves. This is done in Section \ref{SecMain} after recalling Lang's conjecture in Section \ref{SecLang}. 

More precisely: Theorem \ref{ThmMain1} shows that Lang's error term conjecture for curves (Conjecture \ref{ConjLang}) implies Honda's conjecture for elliptic curves. Theorem \ref{ThmMain2} shows that a very weak version of Lang's error term conjecture (namely, Conjecture \ref{ConjWeak}) is still enough to deduce strong rank bounds for elliptic curves. In particular, Corollary \ref{CoroTwist} shows that even this weak version of Lang's error term conjecture implies boundedness of ranks for quadratic twists.

Finally, in Section \ref{SecFF}, we discuss the function field cases of characteristic $0$ and $p>0$. In the former case, we prove the analogue of Lang's error term conjecture for isotrivial elliptic curves.

 Our results can be regarded as evidence towards uniform boundedness of ranks that does not come form probabilistic heuristics on elliptic curves ---as we will recall in the next section, our main evidence comes from analogies with complex analysis. See also Theorem \ref{ThmFF} below.

%%%%%%%%%%%%%%%%%%%%%%%%%%%%%%%%%%%%%%
%%%%%%%%%%%%%%%%%%%%%%%%%%%%%%%%%%%%%%
%%%%%%%%%%%%%%%%%%%%%%%%%%%%%%%%%%%%%%
%%%%%%%%%%%%%%%%%%%%%%%%%%%%%%%%%%%%%%
%%%%%%%%%%%%%%%%%%%%%%%%%%%%%%%%%%%%%%
%%%%%%%%%%%%%%%%%%%%%%%%%%%%%%%%%%%%%%

\section{Lang's conjecture on error terms} \label{SecLang}

 Let $k$ be a number field, $k^{alg}$ a fixed algebraic closure of $k$, and  $X$ a smooth projective variety over $k$. For each line sheaf $\Lcal$ on $X$ it is standard to associate a (logarithmic, normalized over the ground field $k$) height function $h_\Lcal: X(k^{alg})\to \R$ uniquely defined up to a bounded function. When $X=\Pro^n_k$ for some $n$ and $\Lcal=\Ocal(1)$, we simply write $h=h_{\Lcal}$. Given an effective divisor $D$ on $X$ and $S$ a finite set of places of $k$, we also have a proximity function  $m_S(D,-):(X\smallsetminus \supp D)(k^{alg})\to \R$ uniquely defined up to a bounded function. See for instance  \cite{HeightsBG}, \cite{LangNTIII}, or \cite{VojtaCIME} for details.

 Vojta \cite{VojtaThesis, VojtaCIME} formulated deep conjectures in the context of Diophantine approximation and value distribution. We recall here his main conjecture for algebraic points. The canonical sheaf of $X$ is written $\Kcal_X$ and the logarithmic discriminant of $x\in X(k^{alg})$ is defined by $d(x)=[k_x:k]^{-1}\log \Disc(k_x)$, where $k_x\subseteq k^{alg}$ is the residue field of $x$.

\begin{conjecture}[Vojta] \label{ConjVojta} Let $k$ be a number field, $X$ a smooth projective variety over $k$, $S$ a finite set of places of $k$, $D$ an effective normal crossings divisor on $X$, $\Acal$ an ample line sheaf on $X$, $r\ge 1$ an integer, and $\epsilon>0$. There is a proper Zariski closed set $Z\subseteq X$ depending on this data, such that for all but finitely many $x\in (X\smallsetminus Z)(k^{alg})$ with $[k_x:k]\le r$ we have
\begin{equation}\label{EqVojta}
m_S(D,x) + h_{\Kcal_X}(x)< d(x)+ \epsilon h_{\Acal}(x).
\end{equation}
\end{conjecture}

 Vojta \cite{VojtaGen} proved large part of this  when $X$ is a curve, which is the main case  for us. Conjecture \ref{ConjVojta} generalizes several conjectures and theorems in Diophantine approximation. For instance, Roth's theorem can be stated as follows: \emph{Let $S$ be a finite set of places of $k$, $D$ an effective reduced divisor on $\Pro^1_k$ and $\epsilon>0$. For all but finitely many $x\in \Pro^1_k(k)$ we have}
\begin{equation}\label{EqRoth}
m_S(D,x)< (2+\epsilon)h(x).
\end{equation}
Note that this is precisely Conjecture \ref{ConjVojta} for $X=\Pro^1_k$ and $r=1$, i.e. $k$-rational points. 

Lang conjectured \cite{LangAsympt, LangReport, LangTrotter} that \eqref{EqRoth} might be replaced by the following estimate with a better error term ($\log^+t$ means $\log\max\{1,t\}$):
$$
m_S(D,x)< 2h(x) + (1+\epsilon)\log^+h(x).
$$
This was inspired by inspired by Khintchine's theorem, which asserts that an analogous estimate holds for rational approximations of almost all real numbers (in a measure theoretic sense). See also \cite{LangDraft}, p.214 in \cite{LangNTIII}, and paragraph (14.2.8) in \cite{HeightsBG}. 

Motivated by analogies between number theory and value distribution, Lang \cite{LangDuke88} made a similar conjecture for the error term in the Second Main Theorem of Nevanlinna theory in complex analysis. This was later proved by Wong \cite{Wong}, giving valuable evidence for Lang's error term conjecture. See \cite{CherryYe} for a detailed study of  error terms in Nevanlinna theory.

Lang went further and conjectured a similar improvement for the error term in Conjecture \ref{ConjVojta}, namely, that \eqref{EqVojta} should be replaced by
$$
m_S(D,x) + h_{\Kcal_X}(x)< d(x)+ C\log^+ h_{\Acal}(x)
$$
for any given constant $C>1$, cf. p.222 of \cite{LangNTIII}. This improved error term is incorporated in Vojta's formulation of his conjecture in p.198 of \cite{VojtaCIME}, possibly with a larger value of $C$ and restricting to a generic set of algebraic points, see \emph{loc. cit.} for details. 

However, it turns out that when $\dim X\ge 2$  the error term conjecture  \emph{can fail}:  There might be no value of $C$ for which the term $\epsilon h_\Acal(x)$ in \eqref{EqVojta} can be replaced by $C\log^+ h_\Acal(x)$. This is shown by an explicit construction in \cite{LMW}, for rational points in the surface obtained as the blow-up of $\Pro^2_k$ at a point. A key difference between $\dim X\ge 2$ and $\dim X=1$ in Conjecture \ref{ConjVojta} is the exceptional set $Z$: In higher dimensions it  can depend on the  $\epsilon$ appearing in \eqref{EqVojta} (cf. \cite{LMW}), while for curves we may simply take $Z=\supp D$. The structure of $Z$ when $\dim X\ge 2$  is also relevant in the analysis of the parameter $r$ in Conjecture \ref{ConjVojta}, cf. \cite{Levin}.

The problem remains open in the case when $X$ is a curve. Here we state a version weaker than Lang's formulation in \cite{LangNTIII}, as we allow the number $C$ to depend on the geometric data. 
\begin{conjecture}[Error term conjecture for curves] \label{ConjLang} Let $k$ be a number field, $X$ a smooth projective curve over $k$, $D$ an effective reduced divisor on $X$, and $\Acal$ an ample line sheaf on $X$. There is a number $C>0$  depending on the previous data and satisfying the following: 

Let $r$ be a positive integer and let $S$ be a finite set of places of $k$. For all but finitely many algebraic points $x\in X(k^{alg})$ with $[k_x:k]\le r$ we have
\begin{equation}\label{EqLangConj}
m_S(D,x) + h_{\Kcal_X}(x)< d(x)+ C\log^+ h_{\Acal}(x).
\end{equation}
\end{conjecture}

Let us formulate a much weaker conjecture where the number $C$ is allowed to depend on \emph{all} parameters, and the logarithmic discriminant $d(x)$ is replaced by \emph{some} function of the residue field $k_x$. We define the set of fields $\Omega(k,r)=\{L : k\subseteq L\subseteq k^{alg}\mbox{ and }[L:k]\le r\}$.

\begin{conjecture}[Weak error term conjecture] \label{ConjWeak} Let $k$ be a number field, $X$ a smooth projective curve over $k$, $D$ an effective reduced divisor on $X$,  $\Acal$ an ample line sheaf on $X$, $r$ a positive integer, and $S$ a finite set of places of $k$. There is a number $C>0$ and a function $\xi: \Omega(k,r)\to \R$, both possibly depending on all the previous data, such that for all but finitely many algebraic points $x\in X(k^{alg})$ with $[k_x:k]\le r$ we have
$$
m_S(D,x) + h_{\Kcal_X}(x)< \xi(k_x)+ C\log^+ h_{\Acal}(x).
$$
\end{conjecture}

There is another version of Vojta's conjecture involving \emph{truncated counting functions} (a generalization of the radical of an integer), see \cite{VojtaMoreGen}. The refined conjecture includes the $abc$ conjecture as a special case. Using ramified covers, it is shown in \cite{VojtaMoreGen} that Conjecture \ref{ConjVojta} is equivalent to the seemingly stronger version with truncated counting functions. The equivalence, however, does not respect error terms. For instance, for the $abc$ conjecture (formulated in logarithmic form) it is known that the error term cannot be $O(\log^+ h_\Acal(x))$, and in fact, it is at least $\gg_\epsilon (h_\Acal(x))^{1/2-\epsilon}$ for infinitely many examples, cf. \cite{SteTij}. This  difference in the expected  form of error terms in the  truncated and non-truncated versions of Vojta's conjecture was known to Lang (at least in the setting of Roth's theorem compared to the $abc$ conjecture) and has received some attention, see for instance \cite{LangDraft} and \cite{Franken}.

%%%%%%%%%%%%%%%%%%%%%%%%%%%%%%%%%%%%%%
%%%%%%%%%%%%%%%%%%%%%%%%%%%%%%%%%%%%%%
%%%%%%%%%%%%%%%%%%%%%%%%%%%%%%%%%%%%%%
%%%%%%%%%%%%%%%%%%%%%%%%%%%%%%%%%%%%%%
%%%%%%%%%%%%%%%%%%%%%%%%%%%%%%%%%%%%%%
%%%%%%%%%%%%%%%%%%%%%%%%%%%%%%%%%%%%%%

\section{Error terms and ranks} \label{SecMain}

Let us write $\|-\|_2$  for the Euclidean norm  in $\R^m$.
\begin{lemma} \label{LemmaApprox} Let $n\ge 3$ be an integer,  $F:\R^n\to \C$ an $\R$-linear map, and  $H=\ker F$. Assume $H\cap \Z^n=\{0\}$. For certain $c>0$ depending on this data, the following holds:

Let $\delta\in (0,1]$. There is $b\in \Z^n\smallsetminus\{0\}$ satisfying  $\|b\|_2\le c\cdot \delta^{-2/(n-2)}$ and $0<|F(b)|\le\delta$.
\end{lemma}
\begin{proof} For $t >0$ and $m$ positive integer, let $B_m(t)=\{x\in\R^m : \|x\|_2\le t\}$ and let $\omega_m=\Vol B_m(1)$. Since $[-m^{-1/2},m^{-1/2}]^m\subseteq B_m(1)$, we have $\omega_m\ge (m/4)^{-m/2}$. 

Let $P_H: \R^n\to H$ be the orthogonal projection. For $\lambda\ge 1$ and $\epsilon\in (0,1]$, consider the set
$$
X(\lambda, \epsilon):=\{x\in\R^n : P_H(x)\in B_n(\lambda)\mbox{ and }\|x -  P_H(x)\|_2\le \epsilon\}\subseteq \R^n.
$$
 Then $X(\lambda,\epsilon)$ is compact, convex, and symmetric about $0$, and 
 $$
\Vol X(\lambda,\epsilon)=\begin{cases} 
2\epsilon\cdot \lambda^{n-1}\omega_{n-1}&\mbox{ if }\dim_\R(\mathrm{image}(F))=1\\
\pi\epsilon^2\cdot \lambda^{n-2}\omega_{n-2}&\mbox{ if }\dim_\R(\mathrm{image}(F))=2.
\end{cases}
 $$
Hence $\Vol X(\lambda,\epsilon)\ge (n/4)^{-n/2} \epsilon^2\lambda^{n-2}$, as $n\ge 3$. Consider $\lambda_\epsilon :=n^{n/(2n-4)}\epsilon^{-2/(n-2)}\ge 1$ and note that $\Vol X(\lambda_\epsilon,\epsilon)\ge 2^n$. By Minkowski's theorem  there is a non-zero $b_\epsilon\in X(\lambda_\epsilon,\epsilon)\cap \Z^n$  and we observe  that $\|b_\epsilon \|_2\le \lambda_\epsilon+1\le 2 n^{n/(2n-4)} \epsilon^{-2/(n-2)}$.

We have $F(b_\epsilon)\ne 0$ since $H\cap\Z^n=\{0\}$. Let $a_F>0$ be  the operator norm of $F$, then  
$$
0<|F(b_\epsilon)|=|F(b_\epsilon-P_H(b_\epsilon))|\le a_F\cdot \|b_\epsilon -P_H(b_\epsilon)\|_2 \le a_F\cdot \epsilon.
$$
Choosing $\epsilon=\min\{1, a_F^{-1}\}\cdot \delta$, we get the result with $c=2 n^{n/(2n-4)}\cdot \max\left\{1, a_F^{2/(n-2)}\right\}$.
\end{proof}

%%%%%
We now present our results on Lang's error term conjecture  and ranks of elliptic curves.

\begin{theorem}\label{ThmMain1} Let $k$ be a number field and $E$ an elliptic curve  over $k$ with neutral point $x_0$.  If Conjecture \ref{ConjLang} holds for $X=E$ and $D=x_0$, then Conjecture \ref{ConjHonda}  holds for $E$.
\end{theorem}
\begin{proof}
Assume that Conjecture \ref{ConjLang} holds for $X=E$ and $D=x_0$. Note that for a suitable ample sheaf $\Acal$ we can replace the height $h_\Acal$ in \eqref{EqLangConj} by the N\'eron-Tate canonical height $\hat{h}:E(k^{alg})\to \R$. Let $C>0$ be the number provided by Conjecture \ref{ConjLang} with these choices. 

Let $L/k$ be a finite extension, put $r=[L:k]$ and let $S$ be the set of Archimedean places of $k$. To obtain Conjecture \ref{ConjHonda} for $E$, it suffices to prove
\begin{equation}\label{EqClaim}
\rho:=\rk E(L)\le 4Cr +3.
\end{equation}

For the sake of contradiction, suppose $\rho>4Cr+3$. Let $x_1,...,x_\rho\in E(L)$ be $\Z$-linearly independent points and let $c_L$ be an upper bound for the quantity $[k':k]^{-1}\log \Disc(k')$ for all intermediate fields $k\subseteq k'\subseteq L$. In particular,   $d(x)\le c_L$ for all $x\in E(L)$.

Since $\Kcal_E\simeq \Ocal_E$ we can take $h_{\Kcal_E}$ as the zero function. Thus, Conjecture \ref{ConjLang} gives that all but finitely many $x\in E(L)$ satisfy
\begin{equation}\label{EqKey}
m_S(x_0,x) \le c_L + C\log^+ \hat{h}(x).
\end{equation}

Let $\sigma: L\to \C$ be an embedding and let $\tau=\sigma|_k$ be its restriction to $k$. We have an induced embedding $i_\sigma: E(L)\to E_\tau$ where $E_\tau:=(E\otimes_\tau \C)(\C)$. Fix a complex uniformization $\pi: \C\to \C/\Lambda\simeq E_\tau$ for a fixed choice of lattice $\Lambda$. The usual distance in $\C$ induces under $\pi$ a distance function $d_\tau$ on $E_\tau$, hence a distance $d_\sigma$ on $E(L)$ thanks to $i_\sigma$. Fix also $\alpha_1,\ldots,\alpha_\rho\in\C$ lifts of the points $x_j$ for $1\le j\le \rho$ in the sense that $\pi(\alpha_j)=i_\sigma(x_j)$, and observe that the  numbers $\alpha_1,\ldots, \alpha_\rho$ are $\Q$-linearly independent. 

Consider the $\R$-linear map $F:\R^\rho\to \C$ given by 
$$
F(t_1,\ldots,t_\rho)=\sum_{j=1}^\rho \alpha_jt_j.
$$
If in the proximity function $m_S(x_0,-)$ we only keep the contribution of the Archimedean place of $L$ corresponding to $\sigma$, then we see that there is a number $c_1>0$ depending only on the previous choices, such that for each non-zero tuple $b=(b_j)\in\Z^\rho$ the corresponding point $z_b=\sum_j b_jx_j\in E(L)\smallsetminus\{x_0\}$ satisfies 
\begin{equation}\label{EqLower}
m_S(x_0,z_b) +c_1\ge \frac{1}{r}\log^+\left(\frac{1}{d_\sigma(x_0,z_b)}\right)  \ge \frac{1}{r}\log^+\left(\frac{1}{|F(b)|}\right).
\end{equation}

Let $Y\ge 1$ be any real number (which we will later take large) and let $\delta=1/Y$. Note that $\ker(F)\cap \Z^\rho=\{0\}$ because the complex numbers $\alpha_j$ are $\Q$-linearly independent. As $\rho\ge 3$, we can apply Lemma \ref{LemmaApprox} to obtain a non-zero $b_Y \in \Z^\rho$ such that
\begin{equation}\label{EqL1}
|F(b_Y)|\le \delta= Y^{-1} \quad \mbox{ and }\quad \|b_Y \|_2 \le c_2 Y^{2/(\rho-2)}
\end{equation}
for certain $c_2>1$ independent of $Y$.

Let us take the point $z_Y:=z_{b_Y}\in E(L)$. Then \eqref{EqLower} and  \eqref{EqL1} give
$$
m_S(x_0,z_Y)+c_1\ge \frac{1}{r}\log (Y).
$$
On the other hand,  by \eqref{EqL1} and the theory of the N\'eron-Tate canonical height we have
$$
\hat{h}(z_Y)\le c_3 \|b_Y\|_2^2 \le c_4 Y^{4/(\rho-2)}
$$
for some numbers $c_3,c_4>1$ independent of $Y$. Hence, from \eqref{EqKey} we get
$$
\frac{1}{r}\log (Y) -c_1 \le m_S(x_0,z_Y)\le c_L+\frac{4C}{\rho-2}\log(Y) + \log c_4.
$$
The assumption $\rho>4Cr+3$ gives a contradiction for large $Y$, thus proving \eqref{EqClaim}.
\end{proof}
%%%

%%
After some minor notational modifications, the same argument gives:
\begin{theorem}\label{ThmMain2} Let $k$ be a number field and $E$ an elliptic curve  over $k$ with neutral point $x_0$.  If Conjecture \ref{ConjWeak} holds for $X=E$,  $D=x_0$, and a given   $r\ge 1$, then the following holds:

There is a positive integer $B_r$ depending only on the previous data (including $r$) such that for all finite extensions $L/k$ with $[L:k]\le r$ we have $\rk E(L)\le B_r$.
\end{theorem}

In particular, Conjecture \ref{ConjWeak} for $r=2$ implies boundedness of ranks for quadratic twists.

\begin{corollary}\label{CoroTwist} Let $k$ be a number field and $E$ an elliptic curve  over $k$ with neutral point $x_0$.  If Conjecture \ref{ConjWeak} holds for $X=E$,  $D=x_0$, and $r=2$, then Conjecture \ref{ConjTwist} holds for $E$.
\end{corollary}

If in Conjecture \ref{ConjLang} one moreover assumes that  $C$ is uniform for all elliptic curves over $k$, then one would deduce uniform boundedness for ranks of all elliptic curves over $k$.

%%%%%%%%%%%%%%%%%%%%%%%%%%%%%%%%%%%%%%
%%%%%%%%%%%%%%%%%%%%%%%%%%%%%%%%%%%%%%
%%%%%%%%%%%%%%%%%%%%%%%%%%%%%%%%%%%%%%
%%%%%%%%%%%%%%%%%%%%%%%%%%%%%%%%%%%%%%
%%%%%%%%%%%%%%%%%%%%%%%%%%%%%%%%%%%%%%
%%%%%%%%%%%%%%%%%%%%%%%%%%%%%%%%%%%%%%

\section{ Function fields } \label{SecFF}

Let us first consider the complex function field case. For elliptic curves with constant $j$ invariant (the isotrivial case), we can prove a strong form of the analogue of the part of Conjecture \ref{ConjLang} that is relevant for our discussion, with an error term which is \emph{bounded} instead of logarithmic. For the notation regarding Diophantine approximation for function fields, we refer to sections 16 and 28 of \cite{VojtaCIME}; in particular, note that the normalized genus term in the following result  is a geometric analogue of the logarithmic discriminant.
\begin{theorem}\label{ThmFF} Let $K$ be the function field of a smooth complex projective curve $B$ of genus $g_B$ and let $K^{alg}$ be an algebraic closure of $K$. Let $X$ be an elliptic curve over $K$ with constant $j$-invariant and let $x_0$ be the neutral point of $X$. Let $S$ be a finite set of places of $K$. For all $x\in X(K^{alg})\smallsetminus \{x_0\}$ we have that 
$$
m_S(x_0,x) \le \frac{\max\{0,2g_x-2\}}{[K_x:K]}  +\# S
$$
where $K_x$ is the residue field of $x$ and $g_x$ is the genus of the function field $K_x$.
\end{theorem}
\begin{proof} The idea is similar to \cite{SilvermanFF}. See also the proof of Theorem 28.1 in \cite{VojtaCIME}. 

After finite base change, one can reduce to the case when $X$ admits an integral model of the form $B\times X_0$ with $X_0$ an elliptic curve over $\C$, and that $x_0\in X_0(\C)$ is a constant section. The set $S$ can be identified with a finite set of points of $B$. 

There is a smooth projective curve $Y$ with a finite map $\pi:Y\to B$ of degree $r=[K_x:K]$ such that $K_x$ is the function field of $Y$ and the extension $K_x/K$ is induced by $\pi$ under pull-back. Note that $g_Y=g_x$. The point $x\in X(K_x)=\mathrm{Mor} (Y,X_0)$ then corresponds to a morphism $\phi: Y\to X_0$ which we may assume to be non-constant, for otherwise the result is clear ---it is only here that $\max\{0,2g_x-2\}$ is needed in the claimed estimate; the rest of the argument works with this quantity replaced by $2g_x-2$. 

Let $R_\phi$ be the ramification divisor of $\phi$ on $Y$ and for each $p\in Y$ we let $e_p(\phi)$ be the local ramification index. As $X_0$ has genus $1$, the Riemann-Hurwitz formula  gives
$$
2g_Y -2 = \deg R_\phi = \sum_{p\in Y} (e_p(\phi)-1) \ge \left(\sum_{p\in \pi^{-1}(S)}e_p(\phi)\right) - r\cdot \#S.
$$
The result follows since 
$$
m_S(x_0,x) = \frac{1}{r} \sum_{p\in \pi^{-1}(S)\cap \pi^{-1}(x_0)}e_p(\phi).
$$
\end{proof}

 It seems  that the question of boundedness of ranks for non-constant quadratic twists of elliptic curves over $\Q(t)$ with constant $j$-invariant is an open problem, see  \cite{RuSi2}. Along these lines, one can ask: \emph{Let $E$ be an elliptic curve over $\C$. Is there an integer $n_0=n_0(E)\ge 1$ such that for each hyperelliptic curve $X$ over $\C$ we have that $E^{n_0}$ is not an isogenous factor of the Jacobian of $X$?} Relevant examples are constructed in \cite{Paulhus}. Unfortunately, despite Theorem \ref{ThmFF}, our arguments regarding ranks do not seem to apply in this setting as we heavily use an Archimedean place in the number field case.

If instead $K$ is a global function field over a finite field, then  the analogue of Roth's theorem is false \cite{Mahler}, thus, the obvious analogue of Conjecture \ref{ConjLang} fails. This failure is due to constructions involving the Frobenius map. In addition, there is the problem that $K$ does not have Archimedean places. It is appropriate to recall at this point that the ranks of elliptic curves over such a field $K$ are unbounded, even for quadratic twist families in the isotrivial case \cite{ShTa, Ulmer}.

Let us formulate a final conclusion. From the point of view discussed here, \emph{it is conceivable that (un)boundedness of ranks of elliptic curves over number fields and global function fields of positive characteristic are non-analogous phenomena}. In fact, our work suggests that boundedness of ranks is closely related to delicate aspects of Diophantine approximation (namely, error terms),  while it has long been known that height inequalities in Diophantine approximation such as Roth's theorem behave differently in both settings.

%%%%%%%%%%%%%%%%%%%%%%%%%%%%%%%%%%%%%%
%%%%%%%%%%%%%%%%%%%%%%%%%%%%%%%%%%%%%%
%%%%%%%%%%%%%%%%%%%%%%%%%%%%%%%%%%%%%%
%%%%%%%%%%%%%%%%%%%%%%%%%%%%%%%%%%%%%%
%%%%%%%%%%%%%%%%%%%%%%%%%%%%%%%%%%%%%%
%%%%%%%%%%%%%%%%%%%%%%%%%%%%%%%%%%%%%%

\section{Acknowledgments}

Part of these results were obtained during my stay at the Institute for Advanced Study during  the academic year 2015-2016. I thank Zeev Rudnick for valuable feedback on this subject during that time. I also thank Machiel van Frankenhuijsen for providing me a copy of Lang's manuscript \cite{LangDraft}. Comments of Fabien Pazuki, Bjorn Poonen, and Joseph Silverman on preliminary versions of this manuscript are also gratefully acknowledged.

%%%%%%%%%%%%%%%%%%%%%%%%%%%%%%%%%%%%%%
%%%%%%%%%%%%%%%%%%%%%%%%%%%%%%%%%%%%%%
%%%%%%%%%%%%%%%%%%%%%%%%%%%%%%%%%%%%%%


\begin{thebibliography}{9}                                                                               

\bibitem{BKLPR} M. Bhargava, D. Kane, H. Lenstra, B. Poonen, E. Rains, \emph{Modeling the distribution of ranks, Selmer groups, and Shafarevich-Tate groups of elliptic curves}. Camb. J. Math. 3 (2015), no. 3, 275-321. 

\bibitem{HeightsBG}  E. Bombieri, W. Gubler,  \emph{Heights in Diophantine Geometry}. New Mathematical Monographs. Cambridge: Cambridge University Press (2006). doi:10.1017/CBO9780511542879

\bibitem{CherryYe} W. Cherry, Z. Ye,  \emph{Nevanlinna's theory of value distribution. The second main theorem and its error terms}. Springer Monographs in Mathematics. Springer-Verlag, Berlin, 2001. xii+201 pp. ISBN: 3-540-66416-5 
 
\bibitem{Elkies} N. Elkies, \emph{$\mathbb{Z}^{28}$ in $E(\mathbb{Q})$, etc.} Listserv. 3 Apr. 2006. NmbrThry.


\bibitem{Franken} M. van Frankenhuijsen, \emph{About the ABC conjecture and an alternative}.  Number theory, analysis and geometry, 169-180, Springer, New York, 2012. 

\bibitem{Honda} T. Honda, \emph{Isogenies, rational points and section points of group varieties}. Japan. J. Math. 30 1960 84-101.

\bibitem{LangReport} S. Lang, \emph{Report on diophantine approximations}. Bull. Soc. Math. France 93 1965 177-192.

\bibitem{LangAsympt} S. Lang, \emph{ Asymptotic Diophantine approximations}. Proc. Nat. Acad. Sci. U.S.A. 55 1966 31-34. 




\bibitem{LangDuke88}  S. Lang, \emph{ The error term in Nevanlinna theory}. Duke Math. J. 56 (1988), no. 1, 193-218.

\bibitem{LangNTIII} S. Lang, \emph{Number theory. III. Diophantine geometry}. Encyclopaedia of Mathematical Sciences, 60. Springer-Verlag, Berlin, 1991. xiv+296 pp. ISBN: 3-540-53004-5 

\bibitem{LangDraft} S. Lang, \emph{Questions about the error term of diophantine inequalities}. Unpublished. Dated June 30, 2005.

\bibitem{LangTrotter} S. Lang, H. Trotter, \emph{Continued fractions for some algebraic numbers}. J. Reine Angew. Math. 255 (1972), 112-134

\bibitem{Levin} A. Levin, \emph{The exceptional set in Vojta's conjecture for algebraic points of bounded degree}. Proc. Amer. Math. Soc. 140 (2012), no. 7, 2267-2277. 

\bibitem{LMW} A. Levin,  D. McKinnon, J. Winkelmann, \emph{On the error terms and exceptional sets in conjectural second main theorems}. Q. J. Math. 59 (2008), no. 4, 487-498.


\bibitem{Mahler} K. Mahler, \emph{On a theorem of Liouville in fields of positive characteristic}. Canad. J. Math. 1(1949), 397-400



\bibitem{PPVW} J. Park,  B. Poonen, J. Voight, M. Wood. \emph{A heuristic for boundedness of ranks of elliptic curves}. Preprint (2016). arXiv:1602.01431 



\bibitem{Paulhus} J. Paulhus, \emph{Elliptic factors in Jacobians of hyperelliptic curves with certain automorphism groups}.  ANTS X-Proceedings of the Tenth Algorithmic Number Theory Symposium, 487-505, Open Book Ser., 1, Math. Sci. Publ., Berkeley, CA, 2013. 


\bibitem{PoonenICM} B. Poonen, \emph{Heuristics for the arithmetic of elliptic curves}. Preprint (2017). arXiv: 1711.10112

\bibitem{PR} B. Poonen, E. Rains, \emph{Random maximal isotropic subspaces and Selmer groups}. J. Amer. Math. Soc. 25 (2012), no. 1, 245-269

\bibitem{RuSi} K. Rubin, A. Silverberg, \emph{Ranks of elliptic curves in families of quadratic twists}. Experiment. Math. 9 (2000), no. 4, 583-590.

\bibitem{RuSi2} K. Rubin, A. Silverberg, \emph{Rank frequencies for quadratic twists of elliptic curves}. Experiment. Math. 10 (2001), no. 4, 559-569.


\bibitem{Serre} J.-P. Serre,  \emph{Lectures on the Mordell-Weil theorem}.  Third edition. Aspects of Mathematics. Friedr. Vieweg \& Sohn, Braunschweig, 1997. x+218 pp. ISBN: 3-528-28968-6


\bibitem{ShTa} I. Shafarevitch, J.  Tate, \emph{The rank of elliptic curves}. Akad. Nauk SSSR 175 (1967), 770-773

\bibitem{SilvermanFF} J. Silverman, \emph{The $S$-unit equation over function fields}. Math. Proc. Cambridge Philos. Soc. 95 (1984), no. 1, 3-4. 

\bibitem{SteTij} C. Stewart, R. Tijdeman, \emph{On the Oesterl\'e-Masser conjecture}. Monatsh. Math. 102 (1986), no. 3, 251-257.

\bibitem{Ulmer} D. Ulmer, \emph{Elliptic curves with large rank over function fields}, Ann. of Math. (2) 155 (2002), no. 1, 295-315

\bibitem{VojtaThesis} P. Vojta, \emph{Diophantine approximations and value distribution theory}. Lecture Notes in Mathematics, 1239. Springer-Verlag, Berlin, 1987. x+132 pp. ISBN: 3-540-17551-2

\bibitem{VojtaGen} P. Vojta, \emph{A generalization of theorems of Faltings and Thue-Siegel-Roth-Wirsing}. J. Amer. Math. Soc. 5 (1992), no. 4, 763-804. 


\bibitem{VojtaMoreGen} P. Vojta, \emph{A more general abc conjecture}. Internat. Math. Res. Notices 1998, no. 21, 1103-1116. 

\bibitem{VojtaCIME}  P. Vojta, \emph{Diophantine approximation and Nevanlinna theory}. Arithmetic geometry, 111-224, Lecture Notes in Math.  2009, Springer, Berlin (2011).

\bibitem{WDEFGR} M. Watkins, S. Donnelly, N. Elkies, T. Fisher, A. Granville, N. Rogers, \emph{Ranks of quadratic twists of elliptic curves}, Publ. math. de Besançon 2014/2 (2014), 63-98

\bibitem{Wong} P.-M. Wong, \emph{ On the second main theorem of Nevanlinna theory}. Amer. J. Math. 111 (1989), no. 4, 549-583.


\end{thebibliography}
\end{document}